\documentclass[leqno,11pt]{amsart} %leqno is the option to put formula numbers on the left side

\usepackage{amssymb}
\theoremstyle{plain} %text of this environment is typesetted in italics
\newtheorem{theorem}{\indent\sc Theorem}[section]
\newtheorem{lemma}[theorem]{\indent\sc Lemma}
\newtheorem{corollary}[theorem]{\indent\sc Corollary}
\newtheorem{proposition}[theorem]{\indent\sc Proposition}

\theoremstyle{definition} %text of this environment is typesetted in roman letters

\newtheorem{remark}[theorem]{\indent\sc Remark}

\def\C{{\mathbf{C}}}%   \C == \mathbf{C}
\def\R{{\mathbf{R}}}%   \R == \mathbf{R}
\def\L{{\mathbf{L}}}%  \L == \mathbf{L}
\def\Pi{{\mathbf{P}}}%  \Pi == \mathbf{P}
\begin{document}

\title[A survey on Bernstein-type theorems]{A survey on Bernstein-type theorems for \\
entire graphical surfaces} %title of paper and the running head option

\author[Y.~Kawakami]{Yu Kawakami} %first author's name and the running head option

%\dedicatory{***}

%%%%%%%%%%%%%%% footnote %%%%%%%%%%%%%%%%
%\If you use the latest version of amsart.cls, write as follows:
%\subjclass[2020]{ %2020 MSC numbers
%Primary 00; Secondary 00.
%}

\renewcommand{\thefootnote}{\fnsymbol{footnote}}
\footnote[0]{2020\textit{ Mathematics Subject Classification}.
Primary 53A10; Secondary 35J93, 53C24.}
\keywords{%key words and phrases
Bernstein theorem, entire graph, Heinz mean curvature estimate}
\thanks{%acknowledgment of support etc. if any
This work was supported by JSPS KAKENHI Grant Number JP19K03463 and JP23K03086. 
}
%%%%%%%%%%%% Authors addresses %%%%%%%%%%%%%

\address{% first Author
Faculty of Mathematics and Physics, \endgraf
Kanazawa University \endgraf
Kanazawa, 920-1192, \endgraf
Japan
}
\email{y-kwkami@se.kanazawa-u.ac.jp}

%%%%%%%%%%%%%%%%%%%%%%%%%%%%%%%%%%%%%%%%%

\maketitle

\begin{abstract}
We survey Bernstein-type theorems for graphical surfaces in the Euclidean space and the Lorentz-Minkowski space. 
More specifically, we explain several proofs of the Bernstein theorem for minimal graphs in the Euclidean $3$-space. Furthermore, 
we show the Heinz-type mean curvature estimates for graphs in the Euclidean $3$-space and space-like graphs in the Lorentz-Minkowski $3$-space. 
As an application of these estimates, we give Bernstein-type theorems for constant mean curvature graphs in the Euclidean $3$-space 
and constant mean curvature space-like graphs in the Lorentz-Minkowski $3$-space, respectively. We also study Bernstein-type results for 
minimal graphs in the Euclidean $4$-space and the Calabi-Bernstein theorem in the Lorentz-Minkowski $3$-space. 
\end{abstract}

\maketitle

\section{Introduction}
The study of the Bernstein theorem and its generalizations provide an important thrust in the evolution of geometric analysis. 
The classical Bernstein theorem \cite{Be1915, Be1927} in minimal surface theory states that any entire minimal graph in the Euclidean $3$-space $\R^3$ 
must be a plane. In other words, if $\Phi\colon \R^2 \to \R$ is an entire solution of the partial differential equation
\begin{equation}\label{intro-1}
\mathrm{div} \biggl{(} \dfrac{\nabla \Phi }{\sqrt{1+ |\nabla \Phi|^2}} \biggr{)} = 0, 
\end{equation}
then $\Phi$ is a linear function. Here $\nabla \Phi$ means the gradient of $\Phi$. 
This theorem gave rise to the Bernstein conjecture. The conjecture says that any entire minimal graphical hypersurface in the Euclidean $(n+1)$-space $\R^{n+1}$ must be 
a plane, that is, the only entire solution $\Phi\colon \R^n \to \R$ of \eqref{intro-1} is a linear function. The conjecture is affirmative for $n\leq 7$. Its proof is the result of the successive efforts 
by Fleming \cite{Fl1962}, Giorgi \cite{Gi1965}, Almgren \cite{Al1966} and Simons \cite{Si1968}. 
However, the conjecture does not hold for $n\geq 8$, as was shown by Bombieri, Giorgi, Giusti \cite{BGG1969}. See \cite{Os1984} for the history of the Bernstein conjecture.  

Bernstein-type theorems are known other than minimal graphical hypersurfaces in $\R^{n+1}$ for $n\leq 7$.  
For instance, Calabi \cite{Ca1970} (for $n\leq 4$) and Cheng-Yau \cite{CY1976} (for all $n$) showed that any entire maximal space-like graphical hypersurface in the Lorentz-Minkowski 
$(n+1)$-space $\L^{n+1}$ must be a plane. In other words, if $\Psi \colon \R^n \to \R$ is an entire solution of the partial differential equations
\begin{equation}\label{intro-2}
\quad |\nabla \Psi|< 1 \quad \text{and} \quad  \mathrm{div} \biggl{(} \dfrac{\nabla \Psi }{\sqrt{1- |\nabla \Psi|^2}} \biggr{)} = 0,  
\end{equation}
then $\Phi$ is a linear function. This result is called the {\it Calabi-Bernstein theorem}. There exist many other Bernstein-type theorems nowadays.  

The purpose of this paper is to give a survey on Bernstein-type theorems of graphical surfaces (i.e., the case of the dimension of the domain is $2$) in the Euclidean $3$-space and $4$-space ({\S}\ref{chap-E}) and the Lorentz-Minkowski $3$-space ({\S}\ref{chap-L}). The paper is organized as follows: In Section \ref{secE-1}, we explain some proofs of the Bernstein theorem for minimal graphs in $\R^3$. 
In Section \ref{secE-2}, we give the Heinz mean curvature estimate (Theorem \ref{CMC-Heinz}) for graphs in $\R^3$ and the Bernstein-type theorem (Corollary \ref{CMC-Bern}) for constant mean curvature 
graphs in $\R^3$ in reference to \cite[Section 1 in Chapter 2]{Ke2003}. In Section \ref{secE-3}, we state the result (Theorem \ref{R4_Bern}) on a geometric condition where 
an entire minimal graph in $\R^4$ is a complex analytic curve and prove some Bernstein results (Theorems \ref{R4_HHV2009}, \ref{R4_Schoen} and \ref{R4_SL}) by using this result in reference to \cite{HHV2011}. In Section \ref{secL-1}, we explain some duality (Lemma \ref{max-dual}) between minimal graphs in $\R^3$ and maximal space-like graphs in $\L^3$  
and prove the Calabi-Bernstein theorem 
(Theorem \ref{CB-thm1}) by using this duality. In Section \ref{secL-2}, we prove a Heinz-type mean curvature estimate (Theorem \ref{HKKT-thm1}) under an assumption on the gradient bound for space-like graphs in $\L^3$ and show a Bernstein-type theorem (Corollary \ref{Ber-cor6}) for constant mean curvature space-like graphs in $\L^3$ in reference to \cite{HKKT}. 

Finally, the author gratefully acknowledges the useful comments from Jorge Hidalgo and the reviewers.   

\section{Euclidean space}\label{chap-E}

\subsection{Bernstein theorem for minimal graphs in $\R^3$}\label{secE-1}

Let $X\colon \Omega \to \R^n$ be a $C^2$-immersion defined on a 
domain $\Omega \subset \R^2$ into the Euclidean $n$-space $\R^n\, (n\geq 3)$. A surface $X(\Omega) \subset \R^n$ 
is a {\it minimal surface} if its mean curvature $H$ vanishes at every point of $\Omega$. We consider 
in this chapter surfaces in non-parametric form. If the surface defined by 
\begin{equation}\label{mini-nonpara}
X(x, y) = (x, y, f_{3}(x, y), \ldots, f_{n}(x, y)) \in \R^n, \quad f_{k}(x, y) \in C^2 (\Omega, \R) 
\end{equation}
is minimal, we call it a {\it minimal graph} in $\R^n$. 
Then we have the following equations: For $k=3, \ldots, n$,   
\footnotesize
\begin{equation}\label{mini-eq}
\displaystyle \Biggl{(} 1+\sum_{r=3}^{n} \biggl{(} \dfrac{\partial f_{r}}{\partial y} \biggr{)}^{2}\Biggr{)}
\dfrac{\partial^2 f_k}{\partial x^2} -2 \sum_{r=3}^{n} \biggl{(} \dfrac{\partial f_r}{\partial x} \dfrac{\partial f_r}{\partial y} \biggr{)}\dfrac{\partial^2 f_k}{\partial x \partial y}+
\Biggl{(} 1+\sum_{r=3}^{n} \biggl{(} \dfrac{\partial f_{r}}{\partial x} \biggr{)}^{2}\Biggr{)}
\dfrac{\partial^2 f_k}{\partial y^2 } = 0. 
\end{equation}
\normalsize
We call \eqref{mini-eq} the {\it system of minimal surface equations}. For the case $n=3$, 
if we rewrite $f_{3}(x, y)$ as $\Phi (x, y)$, then \eqref{mini-eq} is the following second order nonlinear 
elliptic partial differential equation: 
\begin{equation}\label{mini-eq3}
(1+{\Phi}^2_{y}) \Phi_{xx} -2 \Phi_{x} \Phi_{y} \Phi_{xy} + (1+\Phi^{2}_{x}) \Phi_{yy} =0. 
\end{equation}
By using the gradient $\nabla \Phi := (\Phi_{x}, \Phi_{y})$ of $\Phi$, 
\eqref{mini-eq3} is represented as 
\begin{equation}\label{mini-div}
\mathrm{div} \biggl{(} \dfrac{\nabla \Phi }{\sqrt{1+ |\nabla \Phi|^2}} \biggr{)} =
\dfrac{\partial}{\partial x} \biggl{(} \dfrac{\Phi_x}{W} \biggr{)} +
\dfrac{\partial}{\partial y} \biggl{(} \dfrac{\Phi_y}{W} \biggr{)} = 0,  
\end{equation}
where $W:= \sqrt{1+ |\nabla \Phi|^2}= \sqrt{1+ \Phi^2_{x} + \Phi^2_{y}}$. 
The graph of $\Phi$ is given by 
\[
\Gamma_{\Phi} := \{ (x, y, \Phi (x, y)) \in \R^3 \, |\, (x, y) \in \Omega \}. 
\]  
If $\Omega =\R^2$, then $\Gamma_{\Phi}$ is said to be {\it entire}.   
Here we give three examples of minimal graphs in $\R^3$. 
\begin{enumerate}
\item[(\hspace{.28em}i\hspace{.28em})] The plane: $\Omega =\R^2$, $\Phi (x, y) =ax+by+c\, \,(a, b, c\in \R)$.  
\item[(\hspace{.18em}ii\hspace{.18em})] The helicoid: $\Omega =\R^2 \setminus \{ (x, y) \, | \,  x=0 \}$, $\Phi (x, y) =\tan^{-1}(y/x)$. 
\item[(\hspace{.08em}iii\hspace{.08em})] The upper half of catenoid: $\Omega =\R^2 \setminus \{ (x, y)\, |\, x^2+y^{2}< 1\}$, 
$\Phi (x, y)= \mathrm{cosh}^{-1}\sqrt{x^2 +y^2}$. 
\end{enumerate}

In 1915, Bernstein \cite{Be1915, Be1927} proved the following theorem, which is known as the 
Bernstein theorem.  

\begin{theorem}[Bernstein theorem]\label{min-Bernstein-thm}
The only solution of the minimal surface equation \eqref{mini-eq3} on the whole $(x, y)$-plane $\R^2$ is the trivial solution, that, is, $\Phi$ is a linear function of $x$ and $y$.  
\end{theorem}

The Bernstein theorem can be interpreted geometrically as stating that any entire minimal graph in $\R^{3}$ must be a plane.  

Many proofs of the Bernstein theorem are known (see \cite[Pages 123, 124]{Ni1989}).   
We present here some of them. The first is the Bernstein original proof. 
We give an overview of the proof in reference to \cite{Be1915, Be1927}, \cite{Fa2007} and \cite{Os1984}.
The Bernstein theorem can be seen as a Liouville-type theorem. 
Indeed, Bernstein \cite{Be1915, Be1927} obtained Theorem \ref{min-Bernstein-thm} from the following Liouville-type theorem for solutions of elliptic, not necessarily uniformly elliptic, equations on $\R^2$. 

\begin{theorem}\label{min-Liouville}
Let $a, b, c\colon \R^2 \to \R$ be functions such that the symmetric matrix 
\[
\left(
    \begin{array}{cc}
      a(x, y) & b(x, y) \\
      b(x, y) & c(x, y) 
    \end{array}
  \right)
\]
is positive definite for each $(x, y)\in \R^2$. Let $f\in C^{2} (\R^2, \R)$ be a solution of 
\begin{eqnarray}\label{min-elliptic}
  \left\{
    \begin{array}{l}
      a(x, y) f_{xx} + 2b(x, y) f_{xy} + c(x, y) f_{yy} = 0 \,\, \rm{on} \,\, \mathbf{R}^2, \\
      f(x, y) = o(\sqrt{x^2 +y^2}) \,\, {\rm{as}} \,\, \sqrt{x^2 +y^2}\to +\infty. 
    \end{array}
  \right.
\end{eqnarray}
Then $f$ is constant on $\R^2$.  
\end{theorem} 

If $a\equiv1, b\equiv 0, c\equiv 1$ and $f$ is bounded on $\R^2$, then Theorem \ref{min-Liouville} corresponds to the Liouville theorem of harmonic functions. 
By using Theorem \ref{min-Liouville}, we can prove the Bernstein theorem as follows: 

\begin{proof}[Proof of Theorem \ref{min-Bernstein-thm}]
As it is well-known, any solution of \eqref{mini-eq3} is real analytic. Then a direct calculation shows that 
\[
\varphi_{1} = \tan^{-1}{(\Phi_{x})}, \quad \varphi_{2} = \tan^{-1}{(\Phi_{y})}
\]
are bounded and satisfy  
\[
(1+(\Phi_{y})^2)(\varphi_{i})_{xx} -2\Phi_{x}\Phi_{y}{(\varphi_{i})}_{xy} + (1+(\Phi_{x})^2)(\varphi_{i})_{yy} = 0 \quad (i=1, 2). 
\]
Set $a(x, y) = 1+(\Phi_{y})^2, b(x, y)=  -\Phi_{x}\Phi_{y}, c(x, y)= 1+(\Phi_{x})^2$. Then, by applying  
Theorem \ref{min-Liouville}, we obtain that $\nabla \Phi =(\Phi_{x}, \Phi_{y})$ is constant, that is, $\Phi$ is linear. 
\end{proof}

We give a sketch of the proof of Theorem \ref{min-Liouville}. From the assumption, we have 
\[
f_{xx} f_{yy} -f^2_{xy} \leq 0 
\]   
everywhere on $\R^2$ and the equality holds only at points where $f_{xx}=f_{yy}=f_{xy}=0$ since 
the equation is elliptic. To conclude Theorem \ref{min-Liouville}, Bernstein proved the following lemma. 
However, there exists a gap in the proof, and a complete proof was later given by Hopf \cite{Ho1950} and Mickle \cite{Mi1950}. 

\begin{lemma}\label{min-HM1950}
If $f\in C^2 (\R^2, \R)$ satisfies $f_{xx}f_{yy}-f^{2}_{xy}\leq 0$ (i.e. the Gaussian curvature $K$ of $\Gamma_{f}$ satisfies $K\leq 0$) on $\R^2$ and $f_{xx}f_{yy}-f^{2}_{xy}< 0$ (i.e. $K<0$) 
at some point, then $f(x, y)$ cannot be $o(\sqrt{x^2 +y^2})$ as $\sqrt{x^2 +y^2} \to +\infty$. 
\end{lemma} 

By applying Lemma \ref{min-HM1950}, $f_{xx}f_{yy}-f^2_{xy} = 0$ on $\R^2$, that is, 
$f_{xx}=f_{yy}=f_{xy}=0$ on $\R^2$. Thus $f$ is linear. Moreover, by the second equation of \eqref{min-elliptic}, 
we can show that $f$ is constant.  

The second is a proof by Nitsche. 
Nitsche showed a simple proof of the J\"orgen theorem \cite{Jo1954} in his paper \cite{Ni1957}.  

\begin{theorem}[J\"orgen theorem]\label{min-Jor}
Let $f \in C^2 (\R^2, \R)$ be a solution of the Monge-Amp\`ere equation 
\begin{equation}\label{min-MA}
f_{xx}f_{yy}-f^2_{xy} = 1. 
\end{equation}
Then $f(x, y)$ is a quadratic polynomial function of $x, y$.  
\end{theorem}
See \cite{Ni1957} for a detailed proof. 
We show the proof which derives the Bernstein theorem (Theorem \ref{min-Bernstein-thm}) from the J\"orgen theorem (Theorem \ref{min-Jor}).  
The proof comes from the Heinz observation (see \cite[page 133]{Jo1954}, \cite[Lemma 4.4]{Os1986} and \cite[Section 3 in Chapter 5]{Ko2021}).  

\begin{proof}[Proof of Theorem \ref{min-Bernstein-thm}]
If the function $\Phi$ satisfies \eqref{mini-eq3}, then 
\[
\omega_{1} := \dfrac{1+\Phi^{2}_{x}}{W}\,dx + \dfrac{\Phi_{x}\Phi_{y}}{W}\,dy, \quad 
\omega_{2} := \dfrac{\Phi_{x}\Phi_{y}}{W}\,dx + \dfrac{1+\Phi^{2}_{y}}{W}\,dy 
\]
are closed. By the Poincar\'e lemma, there exist $\xi (x, y), \eta (x, y) \in C^{2} (\R^2, \R)$ 
such that $d\xi =\omega_{1}, d\eta =\omega_{2}$. Moreover, $\omega_{3}:= \xi\,dx+\eta\,dy$ is also closed. By the Poincar\'e lemma, there exists $\zeta (x, y)\in C^{3} (\R^2, \R)$ such that $d\zeta =\omega_{3}$. Then we have 
\[
\zeta_{xx}= \dfrac{1+\Phi^{2}_{x}}{W}, \quad \zeta_{xy}= \dfrac{\Phi_{x} \Phi_{y}}{W}, \quad \zeta_{yy}= \dfrac{1+\Phi^{2}_{y}}{W}.  
\]
Thus $\zeta (x, y)$ satisfies \eqref{min-MA}. By Theorem \ref{min-Jor}, $\zeta_{xx}, \zeta_{xy}, \zeta_{yy}$ are constant, that is, $\Phi_{x}$ and $\Phi_{y}$ are also constant. Hence $\Phi$ is linear.  
\end{proof}

The third is a proof by Heinz. Heinz \cite{He1952} proved the following Gaussian curvature estimate for minimal graphs in $\R^3$.  

\begin{theorem}\label{min-Heinz}
Let $\Delta_{R}$ be the open disk with center at the origin and radius $R\,(>0)$ in $\R^2$, 
and $\Phi (x, y)$ a $C^2$-differentiable function on $\Delta_{R}$. Assume that $\Gamma_{\Phi}$ is a minimal graph on $\Delta_{R}$ in $\R^3$. 
There exists a positive constant $C$ such that the Gaussian curvature $K$ of $\Gamma_{\Phi}$ satisfies 
\begin{equation}\label{mini-curvature-estimate}
|K| \leq \dfrac{C}{R^2}.
\end{equation} 
\end{theorem}

An accessible proof of this theorem can be found in \cite[Chapter 11]{Os1986}. As a corollary of Theorem \ref{min-Heinz}, the Bernstein theorem (Theorem \ref{min-Bernstein-thm}) 
can be proved. Indeed, we obtain $K\equiv 0$ by $R\to \infty$ in \eqref{mini-curvature-estimate}. 
Thus $\Gamma_{\Phi}$ must be a plane because $K\equiv H\equiv 0$. 

The fourth is a proof from value distribution property of the Gauss map suggested by Nirenberg. His idea was to think of the Bernstein theorem in more geometric terms. 
More specifically, the assumption that the surface is an entire graph in $\R^3$ was replaced by the assumption that the surface is complete and its Gauss map omits 
a neighborhood of some point in the $2$-sphere $S^2$. Following this idea, Osserman \cite[Theorem 8.1]{Os1986} proved the following theorem. 

\begin{theorem}\label{min-gauss}
Let $\Sigma$ be an oriented connected $2$-manifold and $X\colon \Sigma \to \R^3$ a complete minimal surface. If the image of its Gauss map is not dense in $S^2$, 
then $X(\Sigma)$ must be a plane. 
\end{theorem}

Theorem \ref{min-gauss} has as an immediate consequence of the Bernstein theorem (Theorem \ref{min-Bernstein-thm}). In fact, an entire minimal graph in $\R^3$ is a complete minimal surface 
whose Gauss map is contained in a hemisphere, hence it must be a plane. 

There exists a close relationship between Theorems \ref{min-Heinz} and \ref{min-gauss} (\cite{KK2024, Ro2001}).         
The optimal result for the size of the set of omitted values of the Gauss maps of nonflat complete minimal surfaces in $\R^3$ is given by Fujimoto \cite[Corollary 1.3]{Fu1988}. 
A geometric interpretation for the Fujimoto result is given in \cite{Ka2013}.  
However, the sharp estimate for the number of omitted values of the Gauss maps of nonflat complete minimal surfaces with finite total curvature in $\R^3$ 
is not known (\cite{Os1964}). In \cite{KKM2008, KW2024}, we give a systematic study on the number of omitted values and the total weight of the totally ramified values of the Gauss maps of 
complete minimal surfaces with finite total curvature in $\R^3$ and $\R^4$. 

\subsection{Bernstein-type theorem for CMC graphs in $\R^3$}\label{secE-2}
In this section, we give the uniqueness theorem for entire constant mean curvature (CMC, for short) graphs in $\R^3$ in reference to \cite[Section 1 in Chapter 2]{Ke2003}. 
Heinz \cite{He1955} showed the following mean curvature estimate for graphs in $\R^3$. 

\begin{theorem}\label{CMC-Heinz}
Let $\Delta_{R}$ be the open disk with center at the origin and radius $R\,(>0)$ in $\R^2$, 
and $\Phi (x, y)$ a $C^2$-differentiable function on $\Delta_{R}$. 
If the mean curvature $H$ of the graph $\Gamma_{\Phi}$ satisfies 
the inequality  
\[
|H|\geq \alpha 
\]
for a positive constant $\alpha$ on $\Delta_{R}$, then the following inequality holds: 
\begin{equation}\label{CMC-eq1}
\alpha \leq \dfrac{1}{R}. 
\end{equation} 
\end{theorem}

\begin{proof}
The mean curvature $H$ of $\Gamma_{\Phi}$ is written as 
\[
H = \dfrac{1}{2}\mathrm{div}\biggl{(} \dfrac{\nabla \Phi }{\sqrt{1+ |\nabla \Phi|^2}} \biggr{)} = \frac{1}{2}\Biggl{\{}
\dfrac{\partial}{\partial x} \biggl{(} \dfrac{\Phi_x}{W} \biggr{)} +
\dfrac{\partial}{\partial y} \biggl{(} \dfrac{\Phi_y}{W} \biggr{)} \Biggr{\}}.  
\]
Take any $R_{1}$ satisfying $0< R_{1}< R$. By the Green theorem, we have 
\begin{equation}\label{Green1}
\displaystyle \iint_{\overline{\Delta}_{R_{1}}}\, 2H \, dxdy = \oint_{x^2 +y^2 =R^2_1} \Biggl{(} -\dfrac{\Phi_{y}}{W}\, dx + \dfrac{\Phi_{x}}{W} \, dy \Biggr{)}. 
\end{equation}
Assume that $H\geq \alpha > 0$ by changing the direction of the normal vector if necessary. 
The left-hand side of \eqref{Green1} becomes 
\[
\displaystyle \iint_{\overline{\Delta}_{R_{1}}}\, 2H \, dxdy\geq  2\pi \alpha R^2_{1}. 
\]
On the other hand, by the Cauchy-Schwarz theorem, the right-hand side of \eqref{Green1} becomes 
\begin{eqnarray}
\oint_{x^2 +y^2 =R^2_1} \Biggl{(} -\dfrac{\Phi_{y}}{W}\, dx + \dfrac{\Phi_{x}}{W} \, dy \Biggr{)} &\leq& \oint_{x^2 +y^2 =R^2_1} \sqrt{\dfrac{|\nabla \Phi|^2}{1+|\nabla \Phi|^2}} (dx^2 +dy^2)^{1/2} \nonumber \\
                                                                                                              &\leq & \oint_{x^2 +y^2 =R^2_1} (dx^2 +dy^2)^{1/2} = 2\pi R_{1}. \nonumber
\end{eqnarray} 
Hence $2\pi \alpha R^2_{1} \leq 2\pi R_{1}$. This proof is completed by letting $R_{1}\to R$. 
\end{proof}

As a corollary of Theorem \ref{CMC-Heinz}, we give the following Bernstein-type theorem for CMC graphs in $\R^3$. 

\begin{corollary}\label{CMC-Bern}
If a graph $\Gamma_{\Phi}$ defined on $\R^2$ satisfies $H\equiv \mathrm{constant}$, then it is a plane. 
In other words, any entire CMC graph in $\R^3$ must be a plane. 
\end{corollary}

\begin{proof}
By Theorem \ref{CMC-Heinz}, $|H|\leq 1/R$ holds. Thus we obtain $H\equiv 0$ by $R\to \infty$. 
From the Bernstein theorem (Theorem \ref{min-Bernstein-thm}), $\Phi$ is a linear function  
and its graph $\Gamma_{\Phi}$ is a plane in $\R^3$. 
\end{proof} 

The inequality \eqref{CMC-eq1} is optimal because the mean curvature of the graph of $\Phi (x, y) =\sqrt{R^2 -x^2-y^2}$ on $\Delta_{R}$ is constant 
and $H=1/R$. 

We remark that Theorem \ref{CMC-Heinz} was extended to the case of graphic hypersurfaces in ${\R}^{n}$ 
by Chern \cite{Ch1965} and Flanders \cite{Fl1966}.  

\subsection{Bernstein-type theorem for minimal graphs in $\R^4$}\label{secE-3} 
In this section, we explain Bernstein-type results for minimal graphs in $\R^4$ in reference to \cite{HHV2011}. 
For $n=4$, \eqref{mini-nonpara} can be represented as 
\[  
X(x, y) = (x, y, \Phi_{1}(x, y), \Phi_{2}(x, y)) \in \R^4, 
\]
where $\Phi_{1}(x, y), \Phi_{2}(x, y) \in C^{2}(\Omega, \R)$. Then the vector-valued map $\Phi\colon \Omega \to \R^2$ is defined by 
\[
\Phi (x, y) := (\Phi_{1} (x, y), \Phi_{2}(x, y)) 
\]
and the graph of $\Phi$ is given by 
\[
\Gamma_{\Phi} := \{ (x, y, \Phi_{1}(x, y), \Phi_{2}(x, y)) \in \R^4 \, |\, (x, y) \in \Omega \}. 
\]
If $\Omega =\R^2$, then $\Gamma_{\Phi}$ is said to be {\it entire}. If the graph of a vector-valued map from $\Omega$ to $\R^2$ is a minimal surface in $\R^4$,  
we call it a {\it minimal graph} in $\R^4$. There exist many entire minimal graphs in $\R^4$ other than the planes. 
For example, if $\Phi_{1}+\mathrm{i}\Phi_{2} \colon \C \to \C$ is holomorphic or anti-holomorphic, then $\Gamma_{\Phi}$ is an entire minimal graph in $\R^4$ and is called a {\it complex analytic curve} 
(\cite[Chapter 2]{Os1986}). There exists an example that is an entire minimal graph in $\R^4$ but not a complex analytic curve. In fact, Osserman \cite[Chapter 5]{Os1986} constructed an entire minimal 
graph in $\R^4$ over the map $\Phi\colon \R^2 \to \R^2$ which is given by 
\begin{equation}\label{R4_Oss}        
\Phi_{1}(x, y) = \dfrac{1}{2}(e^{x}-3e^{-x}) \cos{\dfrac{y}{2}}, \quad \Phi_{2}(x, y) =-\dfrac{1}{2}(e^{x}-3e^{-x}) \sin{\dfrac{y}{2}}. 
\end{equation}
Hasanis, Savas-Halilaj and Vlachos \cite[Theorem 1.1]{HHV2011} obtained a geometric condition where an entire minimal graph in $\R^4$ is a complex analytic curve.  
Here the {\it Jacobian} of $\Phi (x, y) = (\Phi_{1} (x, y), \Phi_{2}(x, y))$ is defined by 
\[
J_{\Phi} := (\Phi_{1})_{x}(\Phi_{2})_{y} -(\Phi_{1})_{y}(\Phi_{2})_{x}.  
\]
For example, the Jacobian of the Osserman example \eqref{R4_Oss} is given by $J_{\Phi} = -(e^{2x}-9e^{-2x})/8$ and $J_{\Phi}$ takes every  
real value. 

\begin{theorem}\label{R4_Bern}
Assume that $\Gamma_{\Phi}$ is an entire minimal graph in $\R^4$ which is not a plane. Then $\Gamma_{\Phi}$ is a complex analytic curve 
if and only if its Jacobian $J_{\Phi}$ does not take every real value. In particular, if $\Gamma_{\Phi}$ is a complex analytic curve, then 
\begin{enumerate}
\item[(\hspace{.18em}i\hspace{.18em})] $J_{\Phi}$ takes every real value in $(0, +\infty)$ or $[0, +\infty)$ when $\Phi_{1}+\mathrm{i}\Phi_{2}$ is holomorphic, and  
\item[(\hspace{.08em}ii\hspace{.08em})] $J_{\Phi}$ takes every real value in $(-\infty, 0)$ or $(-\infty, 0]$ when $\Phi_{1}+\mathrm{i}\Phi_{2}$ is anti-holomorphic. 
\end{enumerate} 
\end{theorem}

See \cite[Chapter 3]{HHV2011} for a proof of this theorem. We here explain Bernstein-type results obtained by applying Theorem \ref{R4_Bern}, which is described in \cite[Chapter 4]{HHV2011}. 
As a corollary of Theorem \ref{R4_Bern}, we can immediately show the following result \cite[Theorem 1.1]{HHV2009} .

\begin{theorem}\label{R4_HHV2009}
Assume that $\Gamma_{\Phi}$ is an entire minimal graph in $\R^4$. If the Jacobian $J_{\Phi}$ of $\Phi$ is bounded, then $\Gamma_{\Phi}$ must be a plane.  
\end{theorem}

\begin{proof}
Since $J_{\Phi}$ is bounded, there exists a positive constant $M$ such that $J_{\Phi}(\R^2) \subset [-M, M]$. By Theorem \ref{R4_Bern}, $\Gamma_{\Phi}$ is a plane. 
\end{proof}

By applying Theorem \ref{R4_Bern}, we can show the following result due to Schoen \cite{Sc1993}. 

\begin{theorem}\label{R4_Schoen}
Let $\Phi = (\Phi_{1}, \Phi_{2}) \colon \R^2 \to \R^2$ be a solution of the system of minimal surface equations \eqref{mini-eq}. 
If $\Phi$ is a diffeomorphism, then $\Phi$ is an affine map. 
\end{theorem}  

\begin{proof}
Since $\Phi$ is a diffeomorphism, it follows that $J_{\Phi} > 0$ if $\Phi$ is orientation preserving, or $J_{\Phi}< 0$ if $\Phi$ is orientation reversing. By Theorem \ref{R4_Bern}, 
the complex function $\Phi_{1}+\mathrm{i}\Phi_{2}$ must be holomorphic or anti-holomorphic. Since the analytic automorphism group of $\C$ is given by 
\[
\mathrm{Aut}(\C) = \{ w=az+b \, | \, a, b\in \C, a\not= 0 \}, 
\]
$\Phi_{1}+\mathrm{i}\Phi_{2}$ is a linear function of $z=x+\mathrm{i}y$. Thus $\Phi$ must be affine.  
\end{proof}

Finally, we give an alternative proof of the following Bernstein-type theorem for special Lagrangian equation due to \cite{Fu1998} and \cite{Yu2002}. 

\begin{theorem}\label{R4_SL}
Let $\varphi \in C^{2} (\mathbf{R}^2, \mathbf{R})$ be a solution of the special Lagrangian 
equation 
\begin{equation}\label{SL_eq}
\cos{\theta}(\varphi_{xx}+\varphi_{yy}) -\sin{\theta}(\varphi_{xx}\varphi_{yy} -\varphi^2_{xy} -1) = 0, 
\end{equation}
where $\theta$ is a real constant. Then $\varphi$ is harmonic or quadratic.  
\end{theorem}

\begin{proof}
Set $\Phi = \nabla \varphi = (\varphi_{x}, \varphi_{y}) \colon \R^2 \to \R^2$. Since $\varphi$ satisfies \eqref{SL_eq}, 
by the Harvey-Lawson result \cite{HL1982}, the graph of $\Phi$ is an entire minimal graph in $\R^4$. Then we have $J_{\Phi} = \Phi_{xx}\Phi_{yy} -\Phi^2_{xy}$. 
Suppose at first that there exists a point $(x_{0}, y_{0})\in \R^2$ such that $J_{\Phi}(x_{0}, y_{0})=1$. Then we obtain $\varphi_{xx}(x_{0}, y_{0}) + \varphi_{yy}(x_{0}, y_{0}) \not= 0$, 
thus we have $\theta =\pi/ 2$ and $J_{\Phi}\equiv 1$. By Theorem \ref{R4_Bern}, $\Phi$ must be affine and thus $\varphi$ is quadratic.   

Suppose next that $J_{\Phi} (x, y)\not= 1$ for any $(x, y)\in \R^2$.  By Theorem \ref{R4_Bern}, we have either $J_{\Phi} >1$ or $J_{\Phi}< 1$. If $J_{\Phi}> 1$ holds, according to 
Theorem \ref{R4_Bern}, $\Phi$ must be affine and so $\varphi$ is quadratic. If $J_{\Phi}< 1$, by virtue of Theorem \ref{R4_Bern}, we obtain that $J_{\Phi}\leq 0$ and 
$\varphi_{x} +\mathrm{i} \varphi_{y}$ is an anti-holomorphic function. Thus, by the Cauchy-Riemann equations, $\varphi$ is harmonic.   
\end{proof}

We remark that Lee \cite{Le2017} gave a short proof of Theorem \ref{R4_SL} by using the J\"orgen theorem (Theorem \ref{min-Jor}).  

\section{Lorentz-Minkowski space}\label{chap-L}

\subsection{Calabi-Bernstein theorem for maximal space-like graphs in $\L^3$}\label{secL-1}
We denote by ${\L}^{n} = (\R^{n}, \langle \,  \,  ,  \, \rangle_{L})\, (n\geq 3)$ the Lorentz-Minkowski $n$-space with the Lorentz metric 
\[
\langle (p_{1}, \ldots , p_{n}), (q_{1}, \ldots , q_{n}) \rangle_{L} := p_{1}q_{1} + \cdots + p_{n-1}q_{n-1} -p_{n}q_{n},  
\]
where $ (p_{1}, \ldots , p_{n}), (q_{1}, \ldots , q_{n}) \in \R^{n}$. Let $\Omega$ be a domain in the Euclidean $2$-space $\R^2$.  
An immersion $X\colon \Omega \to \L^n$ is called {\it space-like} 
if the induced metric $g:= X^{\ast} \langle \,  \,  ,  \, \rangle_{L}$ is positive definite on $\Omega$.  
A space-like surface $X(\Omega) \subset \L^n$ is a {\it maximal surface} if its mean curvature $H$ vanishes at every point of $\Omega$. 
Let $\Psi\colon \Omega \to \R$ be a $C^2$-differentiable function. We consider the graph 
\[
\Gamma_{\Psi} := \{ (x, y, \Psi (x, y)) \in \L^3 \, |\, (x, y) \in \Omega \}
\] 
of $\Psi$. If the graph $\Gamma_{\Psi}$ of $\Psi$ is space-like, its gradient $\nabla \Psi =(\Psi_{x}, \Psi_{y})$ satisfies $|\nabla \Psi|=\sqrt{(\Psi_{x})^2 +(\Psi_{y})^2}< 1$ on $\Omega$. 
Moreover, if a space-like graph $\Gamma_{\Psi}$ in $\L^3$ is a maximal surface (we call it a {\it maximal space-like graph}), then $\Psi$ satisfies the following second order nonlinear 
elliptic partial differential equation:   
\begin{equation}\label{max-ZMC}
(1-\Psi^2_{y})\Psi_{xx} +2\Psi_{x}\Psi_{y}\Psi_{xy} +(1-\Psi^2_{x})\Psi_{yy} = 0. 
\end{equation}
The equation \eqref{max-ZMC} is called the {\it zero mean curvature equation} (ZMC equation, for short) of $\L^3$. 
The equation \eqref{max-ZMC} is also represented as 
\begin{equation}\label{max-div}
\mathrm{div} \left( \dfrac{\nabla \Psi }{\sqrt{1- |\nabla \Psi|^2}} \right) = 
\dfrac{\partial}{\partial x} \left( \dfrac{\Psi_x}{\widetilde{W}} \right) +
\dfrac{\partial}{\partial y} \left( \dfrac{\Psi_y}{\widetilde{W}} \right) =0, 
\end{equation}  
where $\widetilde{W}:=\sqrt{1- |\nabla \Psi|^2}=\sqrt{1-\Psi^{2}_{x}-\Psi^{2}_{y}}$. 

For entire maximal space-like graphs in $\L^3$, the following uniqueness theorem, called the Calabi-Bernstein theorem, is well-known.   

\begin{theorem}[Calabi-Bernstein theorem]\label{CB-thm1}
If a $C^2$-differentiable function $\Psi (x, y)$ on $\R^2$ satisfies $|\nabla \Psi|< 1$ and \eqref{max-ZMC}, 
then $\Psi$ is a linear function of $x$ and $y$.  
\end{theorem}

The Calabi-Bernstein theorem can be interpreted geometrically as stating that any entire maximal space-like graph in $\L^3$ must be a plane.  

Various proofs of the Calabi-Bernstein theorem are known since Calabi \cite{Ca1970} first showed it. 
In this paper, we prove it by using the following result which shows that there exists a duality between minimal graphs in $\R^3$ and 
maximal space-like graphs in $\L^3$. This duality is well-known and is called the {\it Calabi correspondence} (\cite{AP2001, Ca1970, Le2011, Le2013, Sh1956}).  

\begin{lemma}\label{max-dual}
Let $\Omega \subset \R^2$ be a simply-connected domain. 
\begin{enumerate}
\item[(\hspace{.18em}i\hspace{.18em})] If $\Phi \in C^{2} (\Omega, \R)$ satisfies the minimal surface equation \eqref{mini-eq3}, then there exists a function $\Psi\in C^{2} (\Omega, \R)$ satisfying 
$|\nabla \Psi|< 1$, the ZMC equation \eqref{max-ZMC} and 
\begin{equation}\label{max-dual-eq1}
\left(
    \begin{array}{c}
      \Psi_{x} \\
      \Psi_{y}
    \end{array}
  \right)
= \dfrac{1}{\sqrt{1+\Phi^2_{x}+\Phi^2_{y}}} \left(
    \begin{array}{c}
      -\Phi_{y} \\
      \Phi_{x}
    \end{array}
  \right). 
\end{equation}
\item[(\hspace{.08em}ii\hspace{.08em})] If $\Psi \in C^{2} (\Omega, \R)$ satisfies $|\nabla \Psi|< 1$ and the ZMC equation \eqref{max-ZMC}, then there exists a function $\Phi\in C^{2} (\Omega, \R)$ 
satisfying the minimal surface equation \eqref{mini-eq3} and 
\begin{equation}\label{max-dual-eq2}
\left(
    \begin{array}{c}
      \Phi_{x} \\
      \Phi_{y}
    \end{array}
  \right)
= \dfrac{1}{\sqrt{1-\Psi^2_{x}-\Psi^2_{y}}} \left(
    \begin{array}{c}
      \Psi_{y} \\
      -\Psi_{x}
    \end{array}
  \right). 
\end{equation}
\end{enumerate}
\end{lemma}

\begin{proof}
This lemma is proved here in reference to the proof of \cite[Lemma 4.1]{Le2023}. 
\begin{enumerate}
\item[(\hspace{.18em}i\hspace{.18em})] From the assumption, $\Phi$ satisfies \eqref{mini-div}. This implies that 
\[
- \dfrac{\Phi_{y}}{\sqrt{1+\Phi^2_{x}+\Phi^2_{y}}}\,dx +\dfrac{\Phi_{x}}{\sqrt{1+\Phi^2_{x}+\Phi^2_{y}}}\,dy 
\]
is closed. By the Poincar\'e lemma, there exists $\Psi \in C^{2} (\Omega, \R)$ such that 
\[
- \dfrac{\Phi_{y}}{\sqrt{1+\Phi^2_{x}+\Phi^2_{y}}}\,dx +\dfrac{\Phi_{x}}{\sqrt{1+\Phi^2_{x}+\Phi^2_{y}}}\,dy = d\Psi = \Psi_{x}\,dx + \Psi_{y}\,dy. 
\] 
Then it can be easily checked that $\Psi$ satisfies $|\nabla \Psi|< 1$, the ZMC equation \eqref{max-ZMC} and \eqref{max-dual-eq1}. 
\item[(\hspace{.08em}ii\hspace{.08em})] From the assumption, $\Psi$ satisfies $|\nabla \Psi|< 1$ and \eqref{max-div}. 
This implies that 
\[
\dfrac{\Psi_{y}}{\sqrt{1-\Psi^2_{x}-\Psi^2_{y}}}\,dx - \dfrac{\Psi_{x}}{\sqrt{1-\Psi^2_{x}-\Psi^2_{y}}}\,dy 
\]
is closed. By the Poincar\'e lemma, there exists $\Phi \in C^{2} (\Omega, \R)$ such that 
\[
\dfrac{\Psi_{y}}{\sqrt{1-\Psi^2_{x}-\Psi^2_{y}}}\,dx - \dfrac{\Psi_{x}}{\sqrt{1-\Psi^2_{x}-\Psi^2_{y}}}\,dy = d\Phi = \Phi_{x}\,dx+\Phi_{y}\,dy. 
\]
Then it can be easily checked that $\Phi$ satisfies the minimal surface equation \eqref{mini-eq3} and \eqref{max-dual-eq2}.  
\end{enumerate} 
\end{proof}

\begin{remark}
The correspondence between $\Phi$ and $\Psi$ in Lemma \ref{max-dual} implies the duality between the potential and the stream function of 
a Chaplygin gas flow. The correspondence is also refereed to as {\it fluid mechanical duality} (\cite{AUY2020}).   
\end{remark}

\begin{proof}[Proof of Theorem \ref{CB-thm1}]
By (\hspace{.08em}ii\hspace{.08em}) of Lemma \ref{max-dual}, there exists $\Phi \in C^2 (\R^2, \R)$ satisfying \eqref{mini-eq3} and \eqref{max-dual-eq2}. 
Then $\Phi$ is a linear function of $x, y$ from the Bernstein theorem (Theorem \ref{min-Bernstein-thm}).  
By \eqref{max-dual-eq2}, $\Psi_{x}$ and $\Psi_{y}$ are constant, that is, $\Psi$ is also a linear function of $x$ and $y$.  
\end{proof}

\subsection{Bernstein-type theorem for CMC space-like graphs in $\L^3$}\label{secL-2} 

In this section, we show a Bernstein-type theorem for CMC space-like graphs in $\L^3$ 
based on the paper \cite{HKKT}. The uniqueness theorem does not hold for entire CMC space-like graphs 
in $\L^3$. In fact, a hyperboloid is an entire CMC space-like graph in $\L^3$ which is not 
a plane. Many other entire CMC space-like graphs are constructed by Treibergs \cite{Tr1982}. 
In \cite[Theorem 2.1]{HKKT}, we give the following Heinz-type mean curvature estimate under an assumption on the gradient bound for space-like graphs in $\L^3$. 

\begin{theorem}\label{HKKT-thm1}
Let $\Delta_{R}$ be the open disk with center at the origin and radius $R\,(>0)$ in $\R^2$, 
and $\Psi (x, y)$ a $C^2$-differentiable function on $\Delta_{R}$. 
Suppose that there exist constants $M>0$ and $k\in \R$ such that 
\begin{equation}\label{HKKT-eq1}
\dfrac{|\nabla \Psi|}{\sqrt{|1-|\nabla \Psi|^2|}}\leq M  (x^2 +y^2)^{k}.  
\end{equation}  
Assume that $\Gamma_{\Psi}$ is a space-like graph of $\Psi$ in $\L^3$.  
If the mean curvature $H$ of $\Gamma_{\Psi}$ satisfies the inequality  
\[
|H|\geq \alpha 
\]
for a positive constant $\alpha$ on $\Delta_{R}$, then the following inequality holds: 
\begin{equation}\label{HKKT-eq2}
\alpha \leq MR^{2k-1}. 
\end{equation}
\end{theorem}

\begin{proof}
Since $\Gamma_{\Psi}$ is space-like, $|\nabla \Psi| <1$ holds. 
Then the mean curvature $H$ of $\Gamma_{\Psi}$ is written as 
\[
H = \dfrac{1}{2}\mathrm{div}\biggl{(} \dfrac{\nabla \Psi }{\sqrt{1- |\nabla \Psi|^2}} \biggr{)} = \frac{1}{2}\Biggl{\{}
\dfrac{\partial}{\partial x} \biggl{(} \dfrac{\Psi_x}{\widetilde{W}} \biggr{)} +
\dfrac{\partial}{\partial y} \biggl{(} \dfrac{\Psi_y}{\widetilde{W}} \biggr{)} \Biggr{\}}.  
\]
We take any $R_{1}$ satisfying $0< R_{1}< R$. By the Green theorem, we have  
\begin{equation}\label{Green2}
\displaystyle \iint_{\overline{\Delta}_{R_{1}}}\, 2H \, dxdy = \oint_{x^2 +y^2 =R^2_1} \Biggl{(} -\dfrac{\Psi_{y}}{\widetilde{W}}\, dx + \dfrac{\Psi_{x}}{\widetilde{W}} \, dy \Biggr{)}. 
\end{equation}
Assume that $H\geq \alpha > 0$ by changing the direction of the normal vector. 
Then the left-hand side of \eqref{Green2} becomes 
\[
\displaystyle \iint_{\overline{\Delta}_{R_{1}}}\, 2H \, dxdy \geq 2\pi \alpha R_{1}^{2}. 
\] 
On the other hand, by the Cauchy-Schwarz theorem, the right-hand side of \eqref{Green2} becomes 
\begin{eqnarray}
\oint_{x^2 +y^2 =R^2_1} \Biggl{(} -\dfrac{\Psi_{y}}{\widetilde{W}}\, dx + \dfrac{\Psi_{x}}{\widetilde{W}} \, dy \Biggr{)} &\leq& \oint_{x^2 +y^2 =R^2_1} \dfrac{|\nabla \Psi |}{\sqrt{1-|\nabla \Psi|^2}} (dx^2 + dy^2)^{1/2} \nonumber \\
                                                                                                              &\leq& MR_{1}^{2k} \oint_{x^2 +y^2 =R^2_1} (dx^2 +dy^2)^{1/2}  \nonumber \\
&=& 2\pi M R_{1}^{2k+1}. \nonumber
\end{eqnarray} 
Hence $2\pi \alpha R^2_{1} \leq 2\pi M R_{1}^{2k+1}$ holds. This proof is completed by letting $R_{1}\to R$. 
\end{proof}

By considering the above argument on a relatively compact domain in $\R^2$, we give the following result which was obtained by Salavessa \cite[Theorem 1.5]{Sa2008} for space-like graphs in $\L^3$.  

\begin{proposition}\label{Sal-prop3}
Let $\Omega$ be a relatively compact domain (i.e. its closure $\overline{\Omega}$ is compact) of ${\R}^2$ with smooth boundary $\partial \Omega$.  
Assume that $\Gamma_{\Psi}$ is a space-like graph in $\L^3$ of a $C^2$-differentiable function $\Psi$. Set $m_{\Omega}:= \max_{\overline{\Omega}} |\nabla \Psi|$. For the mean curvature $H$ of $\Gamma_{\Psi}$, we have 
\[
\displaystyle \min_{\overline{\Omega}} |H| \leq \dfrac{1}{2} \dfrac{m_{\Omega}}{\sqrt{1- (m_{\Omega})^2}} \dfrac{L (\partial \Omega)}{A (\overline{\Omega})}. 
\] 
Here $L(\partial \Omega)$ is the length of $\partial \Omega$ and $A(\overline{\Omega})$ is 
the area of $\overline{\Omega}$.  
\end{proposition}

As a corollary of Theorem \ref{HKKT-thm1}, we give the following vanishing theorem of mean curvature for entire CMC space-like graphs in $\L^3$. 

\begin{corollary}\label{vani-cor4}
Assume that $\Gamma_{\Psi}$ is an entire space-like graph in $\L^3$ of a $C^2$-differentiable function $\Psi (x, y)$. If the entire graph $\Gamma_{\Psi}$ has constant mean curvature and there exist constants $M>0$ and $\varepsilon >0$ such that 
\begin{equation}\label{vani-eq5}
\dfrac{|\nabla \Psi|}{\sqrt{1-|\nabla \Psi|^2}}\leq M \left( x^2 + y^2  \right)^{(1/2) - \varepsilon}
\end{equation}
on ${\R}^2$, then its mean curvature must vanish everywhere.   
\end{corollary} 

\begin{proof}
By Theorem \ref{HKKT-thm1}, the mean curvature $H$ of $\Gamma_{\Psi}$ satisfies 
\[
|H|\leq \dfrac{M}{R^{2\varepsilon}}
\]
on $\Delta_{R}$. We obtain $H\equiv 0$ by letting $R\to +\infty$. 
\end{proof}

We give a geometric interpretation for $|\nabla \Psi|/ \sqrt{1-|\nabla \Psi|^2}$. 
When the graph is space-like (i.e., $|\nabla \Psi|< 1$), 
\[
\nu\, (=\nu (x, y)) = \dfrac{1}{\sqrt{1-|\nabla \Psi|^2}} (\Psi_{x}, \Psi_{y}, 1)
\]
is the time-like unit normal vector field of ${\Gamma}_{\Psi}$. Since $e_{3}=(0, 0, 1) \in \L^{3}$ is also time-like, there exists a unique real-valued function $\theta\, (=\theta (x, y))\geq 0$ such that  
$
\langle \nu, e_{3} \rangle_{L} = -\cosh{\theta}. 
$
This function $\theta$ is called the {\it hyperbolic angle} between $\nu$ and $e_{3}$ (see \cite{On1983}). By simple calculation, we have 
\[
\sinh{\theta} = \dfrac{|\nabla \Psi|}{\sqrt{1-|\nabla \Psi|^2}}. 
\]

By Theorem \ref{CB-thm1} and Corollary \ref{vani-cor4}, we obtain the following Bernstein-type theorem for entire CMC space-like graphs in $\L^3$. 

\begin{corollary}\label{Ber-cor6}
If an entire space-like graph $\Gamma_{\Psi}$ in $\L^3$ of a $C^2$-differential function $\Psi (x, y)$ has constant mean curvature and there exist constants $M>0$ and $\varepsilon >0$ such that 
\begin{equation}\label{hyp-eq7}
\sinh{\theta} \leq M (x^2 + y^2 )^{(1/2)-\varepsilon}
\end{equation}
on $\R^2$, then it must be a plane. Here $\theta$ is the hyperbolic angle between $\nu$ and $e_{3}$. 
\end{corollary}

Corollary \ref{Ber-cor6} is optimal, since there exists an example which is not congruent to a plane and satisfies \eqref{hyp-eq7} for $\varepsilon =0$. In fact, 
the function $\Psi$ given by 
\[
\Psi (x, y) = \left( x^2 + y^2 + \dfrac{1}{H^2} \right)^{1/2} \quad (H>0)
\]
is not linear and its graph is a hyperboloid and has constant mean curvature $H$. 
Moreover, it satisfies 
\[
\sinh{\theta} = H \sqrt{x^2 +y^2},  
\] 
which is the equality in \eqref{hyp-eq7} for $\varepsilon =0$. 

%%%%%%%%%%%%%%%%%%%%%%%%%%%%%%%%%
% References
%%%%%%%%%%%%%%%%%%%%%%%%%%%%%%%%%

\end{document}